\documentclass[11pt]{amsart}
\usepackage{amssymb,amsmath,amsthm}
\usepackage{graphicx}
\usepackage[final]{hyperref}
\usepackage[a4paper,centering]{geometry}
\usepackage{color}

\newtheorem{theorem}{Theorem}
\newtheorem{proposition}[theorem]{Proposition}
\newtheorem{lemma}[theorem]{Lemma}

\theoremstyle{definition}

\theoremstyle{remark}

\def\R{\mathbb{R}}

\def\e{\varepsilon}
\def\Om{\Omega}

\newcommand{\eigeng}{\lambda_{p,\gamma}}
\DeclareMathOperator{\dive}{div}

\begin{document}

\title[Faber--Krahn for homogeneous normalized $p$-Laplacian type operators]
{Faber--Krahn inequality for homogeneous normalized $p$-Laplacian type operators}

\author[G.~Crasta, I.~Fragal\`a]{Graziano Crasta, Ilaria Fragal\`a}

\address[Graziano Crasta]{Dipartimento di Matematica ``G.\ Castelnuovo'',
Sapienza University of Rome\\
P.le A.\ Moro 5 -- 00185 Roma (Italy)}
\email{graziano.crasta@uniroma1.it}

\address[Ilaria Fragal\`a]{Dipartimento di Matematica, Politecnico di Milano\\
Piazza Leonardo da Vinci, 32 -- 20133 Milano (Italy)}
\email{ilaria.fragala@polimi.it}


\date{October 16, 2026}

\begin{abstract}
We establish a sharp Faber–Krahn inequality for the normalized (or game-theoretic) $p$-Laplacian, proving that, for $p\ge 2$, the ball uniquely minimizes the principal Dirichlet eigenvalue among bounded connected Lipschitz domains of prescribed volume satisfying a uniform exterior sphere condition. More generally, we prove the same inequality for the family of homogeneous operators $|Du|^\gamma\Delta_p^N u$, with $p>1$ and $-1<\gamma\le p-2$. Unlike the classical $p$-Laplacian, these operators lack a variational structure, preventing the use of standard symmetrization arguments. Our proof overcomes this difficulty through a level-set approach: it relies on a flux identity for the eigenvalue, for which a variational interpretation re-emerges in the special case where the domain is a ball.

\end{abstract}
\maketitle

\section{Introduction and statement of the result}

In 1877, Rayleigh conjectured that, among all membranes of prescribed area,
the disk has the lowest fundamental frequency. 
The conjecture
was proved in 1923 by Faber \cite{Faber1923},  and  independently  two years later by Krahn
\cite{Krahn1925}, who subsequently extended the result to higher
dimensions \cite{Krahn1926}.  Nowadays the inequality is 
named after Faber--Krahn, and   is 
classical for the $p$-Laplacian as well, see e.g.~\cite{Bhattacharya1999}. 
In modern terms it states that,  among bounded open subsets of $\R ^n$ of given $n$-dimensional measure, 
for any $p \in (1, +\infty)$
the first  Dirichlet eigenvalue $\lambda _p$ of  the $p$-Laplacian  
is minimized by balls. 

The most direct proof  relies on the variational nature of the eigenvalue as infimum of the Rayleigh  quotient, 
which makes possible to argue via symmetrization (see \cite{Kawohl,KawohlFridman}):
given a  function 
$u\in W_0^{1,p}(\Omega)$, 
the symmetric decreasing rearrangement  
preserves its
$L ^p$ norm,  while it diminishes that  of $Du$. Hence,   if $B$ is a ball with the same volume as $\Om$, 
 \begin{equation}\label{f:symm} 
\lambda_{p}(\Omega)=
\inf_{u\in W_0^{1,p}(\Omega)}
\frac{\int_\Omega |D u|^p\,dx}
     {\int_\Omega |u|^p\,dx} \geq 
     \inf_{u^*\in W_0^{1,p}(B)}
     \frac{\int_B |D u^* |^p\,dx}
     {\int_B |u^*|^p\,dx} 
 = \lambda_{p}(B) \,.
\end{equation} 
 This variational structure is
precisely what is lost for the normalized, or game-theoretic,
$p$-Laplacian, 
which arises naturally in stochastic
game interpretations and in asymptotic mean-value formulas for $p$-harmonic
functions,  see \cite{PeresSheffield,ManfrediParviainenRossi,KawohlNormalized}
and the references therein.  
 
It is defined by
\begin{equation}\label{f:def-normalized}
\Delta_p^Nu  := \frac{1}{p|Du|^{p-2} } \Delta_pu 
 =\frac1p\Big (\Delta u+(p-2)\big\langle D^2u\frac{Du}{|Du|},\frac{Du}{|Du|}\big \rangle  \Big )\,.
\end{equation}
Since $\Delta_p^N$ is not in divergence form, depends on the direction of the
gradient, and is singular at critical points,   its spectral  
 theory is markedly
different from the variational theory of $\Delta_p$. 
In particular, eigenfunctions must be understood in the  viscosity sense. 
In  our joint work
\cite{CrastaFragalaKawohl} with B.~Kawohl, we established the 
simplicity of the first Dirichlet eigenvalue of $\Delta_p^N$  and a
non-sharp lower bound in terms of the measure of the domain, 
leaving the corresponding Faber--Krahn inequality as an open problem.

The aim of this work is to attack this problem, adopting a more general perspective. In fact, we consider a wider family of homogeneous operators
which has attracted increasing attention in recent years 
\cite{ImbertSilvestre2013,AttouchiRuosteenoja,WangYinJiang2025,AlcantaraSantos2026,AlcantaraSantosUrbano2026}

For any  $p\in (1, + \infty)$  and any $\gamma > -1$, we consider
the  Dirichlet eigenvalue problem
\begin{equation}\label{f:eigen-normalized}
-|Du|^\gamma\Delta_p^N u
=\eigeng(\Omega)|u|^{\gamma} u 
\quad\hbox{in }\Omega,
\qquad
u=0\quad\hbox{on }\partial\Omega \,,
\end{equation}
where the power $|u|^{\gamma} u $  is dictated by homogeneity.

The differential operator in \eqref{f:eigen-normalized} corresponds to $\Delta _p$ and  to $\Delta _p ^N$  
for $\gamma=p-2$ and 
$\gamma=0$ respectively.  For any $\gamma > -1$, it
falls into the setting of fully nonlinear homogeneous operators studied by Birindelli and Demengel 
in \cite{BirindelliDemengel2007,BirindelliDemengel2009} and in our previous paper
\cite{CrastaFragala2020}.
The precise definitions of viscosity solution and  principal-eigenvalue are postponed for convenience to Section~\ref{s:preliminaries}. 
Here 
we limit ourselves to recall  that, by  \cite[Theorem 19]{CrastaFragala2020},   problem
\eqref{f:eigen-normalized} admits  a positive
principal eigenfunction  continuous up to $\partial\Omega$, 
 provided $\Omega$ belongs to the class 
 \begin{equation}
\label{f:sets}
\mathcal A_n:= \left\{ 
\begin{matrix}
\text{open bounded connected Lipschitz domains of $\R^n$} 
\\ 
\text{satisfying a uniform exterior sphere condition}
\end{matrix}
\right\}\,.
\end{equation}
Thus, under these regularity assumptions on $\Omega$, asking  whether $\lambda _{p, \gamma}$ 
satisfies a Faber-Krahn type inequality is a meaningful spectral  question for a
non-variational viscosity operator. 

Our main result is the following: 

\begin{theorem}[Faber--Krahn inequality]\label{t:FK-gamma}
Let $n\ge2$ and $p\in (1, +\infty)$.  Assume that
\[
-1<\gamma\le p-2\,.
\]
 
For any $\Omega\in\mathcal A_n$, if $\Omega^*$ is  a ball such that
$|\Omega|=|\Omega^*|$,
then
\begin{equation}\label{f:FK-main}
\eigeng(\Omega)\ge \eigeng(\Omega ^*)\,
\end{equation}
with equality if and only if $\Omega= \Omega ^*$ (up to rigid motions).
\end{theorem}

Notice that, for $\gamma=0$, the theorem settles the Faber--Krahn problem for the normalized $p$-Laplacian
posed in \cite{CrastaFragalaKawohl} in the range $p\ge2$, while the endpoint $\gamma=p-2$ is already contained in the classical variational
Faber--Krahn theorem for the $p$-Laplacian,  the genuinely new case being  
the strict range $-1<\gamma<p-2$.

Faber--Krahn type results for fully nonlinear
elliptic operators are comparatively scarce.
Among the few contributions in this direction,
Birindelli and Leoni \cite{BirindelliLeoni2014}
proved that symmetry minimizes 
the principal eigenvalue of the Pucci
supremum operator within a special class of
planar domains. Our result provides instead a Faber--Krahn
inequality, including rigidity in the equality
case, for a family of homogeneous fully
nonlinear operators on bounded connected Lipschitz domains
satisfying a uniform exterior sphere
condition, without relying on a
global variational characterization of the
principal eigenvalue. 

In this respect, a relevant precedent is provided by the work of Hamel, Nadirashvili and Russ \cite{HamelNadirashviliRuss2007,HamelNadirashviliRuss2011}, who established isoperimetric inequalities for principal eigenvalues of linear second-order elliptic operators in divergence form, with possibly non-symmetric coefficients and lower-order terms, without relying on a Rayleigh quotient characterization. Their method is based on an ad hoc rearrangement of the principal eigenfunction and of the coefficients of the operator. Our approach is different, although it shares with theirs the use of level-set geometry and flux integrals associated with the principal eigenfunction to obtain a radial comparison. 

Let us briefly explain the mechanism of our proof. The starting point is the observation that a Rayleigh quotient characterization of $\lambda_{p,\gamma}$ survives in the case of balls. Indeed, the ordinary differential equation satisfied by the principal eigenfunction of the ball $B_R$ is precisely the Euler–Lagrange equation of the minimization problem
\begin{equation}\label{f:var} 
p' (\gamma+1)  \eigeng(B_R)
= \inf_{\substack{\phi\not\equiv0\\ \phi(R)=0}}
\frac{\displaystyle\int_0^Rr^\delta|\phi'|^{\gamma+2}\,dr}
{\displaystyle\int_0^Rr^\delta|\phi|^{\gamma+2}\,dr}\,, \qquad \text{ with } \delta = \frac{(n-1)(\gamma+1)}{p-1 }\,. 
\end{equation} 
The power-type weight $r^\delta$ emerging in this radial problem
can be viewed as corresponding to the effective radial dimension  $N_{\mathrm{eff}} = 1 + \delta $ (which equals $n$ at the endpoint $\gamma=p-2$).

In view of \eqref{f:var}, in order to prove the Faber-Krahn inequality for  $\eigeng$, it is enough 
to show that the radial profile $v = v ( r)$ of the symmetric decreasing rearrangement of the first eigenfunction $u$ of $\Om$  satisfies 
\begin{equation}\label{f:required}  p' (\gamma+1)  \eigeng(\Omega) \geq 
\frac{\displaystyle\int_0^Rr^\delta|v'|^{\gamma+2}\,dr}
{\displaystyle\int_0^Rr^\delta|v|^{\gamma+2}\,dr} \,.
 \end{equation}  
Indeed, the inequality $\eigeng(\Omega)  \geq \eigeng(B_R)$ follows from \eqref{f:var} and \eqref{f:required}, simply by taking  $v$  
as a test function in the  one-dimensional radial problem.  

To prove \eqref{f:required}, we exploit a  weak formulation of the eigenvalue problem due to Attouchi and Ruosteenoja 
\cite{AttouchiRuosteenoja} (see Proposition 
\ref{p:weak-formulation} below). 
Testing it with truncations of the eigenfunction yields a  representation formula for the first eigenvalue of $\Om$ in terms of the level sets of  its eigenfunction $u$: 
 \begin{equation}\label{f:evaluate}  
\int_{ \{ u = t \} }|Du|^{p-1}\,d\mathcal H^{n-1} 
=p\eigeng(\Omega)
\int_{\{ u > t \} }u^{\gamma+1}|Du|^{p - 2 - \gamma } \,dx \qquad \text{ for a.e. } t\in (0, 1)\, .
\end{equation}  
Using this level-set representation, we derive a differential inequality for $v$
which in turn yields \eqref{f:required}, by analyzing the flux functional
 $$F(t):= \int_{ \{ u = t \} }|Du|^{p-1}\,d\mathcal H^{n-1}  \,.$$ 
More precisely, we combine a geometric lower bound for $F$, obtained via the application of the
 coarea formula and the isoperimetric inequality, with  a 
differential estimate for $F$, obtained through a careful analysis of
the generalized inverse of the distribution function
$|\{u>t\}|$.  The assumption $\gamma\le p-2$ is motivated precisely by the interpolations used in this
step.  We point out that, at the endpoint $\gamma=p-2$, the geometric lower bound
alone is sufficient. In this case, the argument follows the idea underlying
Krahn's original proof, as revisited in \cite{Daners2011}, and appearing
in later developments, including Talenti's seminal work \cite{Talenti1976}
and Daners' approach to Robin problems \cite{Daners2006, BucurDaners2010}.

 Finally, if equality holds, all inequalities involved must become
equalities.  In particular, equality in the isoperimetric inequality forces almost
every positive superlevel set of $u$ to be a ball, and nestedness then yields
the rigidity statement for $\Omega$.

The paper is organized as follows. Section~\ref{s:preliminaries}
collects some preliminaries on viscosity solutions to the
eigenvalue problem and their weak reformulation, together with
some useful properties of generalized inverse functions.
In Section~\ref{s:1d}, we explain the key reduction of the
Faber--Krahn inequality to a differential inequality for the
Schwarz rearrangement of the principal eigenfunction.
Section~\ref{s:proof-radial} is entirely devoted to the proof
of this inequality. Finally, Section~\ref{s:rigidity} deals
with the equality case and establishes the uniqueness of
balls as minimizers.

\section{Preliminaries}\label{s:preliminaries}

\subsection{On viscosity and weak solutions to the eigenvalue problem }\label{s:weak}

The differential operator $-|Du|^\gamma\Delta_p^N u$ appearing in 
\eqref{f:eigen-normalized} can be written as
\begin{equation}\label{f:eigen-Fgamma}
-|Du|^\gamma\Delta_p^N u = \mathcal F_{p,\gamma}(Du,D^2u)
\end{equation}
where, 
for every $\xi\in\mathbb R^n\setminus\{0\}$ and every symmetric $(n\times n)$ matrix $X$, 
\begin{equation}\label{f:def-Fgamma}
\mathcal F_{p,\gamma}(\xi,X)
:=-\frac1p|\xi|^\gamma
\Big(\operatorname{tr}X+(p-2)
\big\langle X\frac{\xi}{|\xi|},\frac{\xi}{|\xi|}\big\rangle\Big).
\end{equation}

  By a positive
Dirichlet eigenfunction of $\mathcal F_{p,\gamma}$ associated with $\lambda$ we mean a function
\[
u\in C(\overline\Omega),\qquad u>0\ \hbox{in }\Omega,
\qquad u=0\ \hbox{on }\partial\Omega,
\]
which solves
\begin{equation}\label{f:pb}
\mathcal F_{p,\gamma}(Du,D^2u)=\lambda u^{\gamma+1}
\quad\hbox{in }\Omega
\end{equation}

in the following viscosity sense.

For $\gamma\ge0$, we adopt the standard
Crandall--Ishii--Lions definition
\cite{CrandallIshiiLions}, based on the
semicontinuous envelopes of the operator
defined in \eqref{f:def-Fgamma}.
More precisely, an upper semicontinuous
function $u$ is a viscosity subsolution
of \eqref{f:pb} if, whenever
$\phi\in C^2(\Omega)$ touches $u$ from
above at $x_0\in\Omega$, one has
\[
(\mathcal F_{p,\gamma})_*
(D\phi(x_0),D^2\phi(x_0))
\le \lambda u(x_0)^{\gamma+1}.
\]
Supersolutions are defined analogously,
using test functions touching from below
and the reverse inequality for the upper
semicontinuous envelope
$(\mathcal F_{p,\gamma})^*$.
The envelopes are taken jointly with
respect to the gradient and Hessian
variables. For $\gamma>0$, the operator
extends continuously to $\xi=0$ by
setting $\mathcal F_{p,\gamma}(0,X)=0$;
for $\gamma=0$, it has a bounded
discontinuity at $\xi=0$.

For $-1<\gamma<0$, we adopt instead
the Birindelli--Demengel viscosity
definition
\cite{BirindelliDemengel2007,
BirindelliDemengel2009}.

According to this definition, a lower
semicontinuous function $u$ is a viscosity
supersolution of \eqref{f:pb} if the
following conditions hold at every
$x_0\in\Omega$:
if $u$ is constant with value $c$ in
a neighborhood of $x_0$, then
\[
0\ge \lambda c^{\gamma+1};
\]
otherwise, every $C^2$ test function
$\phi$ touching $u$ from below at $x_0$,
with $D\phi(x_0)\ne0$, satisfies
\[
\mathcal F_{p,\gamma}
(D\phi(x_0),D^2\phi(x_0))
\ge \lambda u(x_0)^{\gamma+1}.
\]
Subsolutions are defined analogously,
using test functions touching from above
and reversing the inequalities.

The equivalence of this definition
with the singular viscosity definition
based on test functions with isolated
critical points is established in
\cite[Proposition~2.4]{AttouchiRuosteenoja}.

  The principal eigenvalue of $\mathcal F _{p,\gamma}$  is understood in the sense of \cite{BirindelliDemengel2007,BirindelliDemengel2009}, namely 
\[
\lambda _{p, \gamma} (\Omega):= \sup \Big \{ \lambda \in \R \ :\ \exists u >0 \text{ in } \Omega \text{ viscosity super-solution to the pde in } \eqref{f:pb}\Big \} \,.
\]
 Then, 
  the existence of a
positive eigenfunction  on the class
$\mathcal A_n$, along with its boundary continuity,
follows from \cite[Theorem 19 and Proposition 17]{CrastaFragala2020}.  

Since the equation is homogeneous in
$u$, from now on we normalize the principal eigenfunction by
\begin{equation}\label{f:max-normalization}
\max_{\overline\Omega}u=1.
\end{equation}

In the strict range $-1<\gamma<p-2$,
at every point where $Du\ne0$, the identity
\[
\Delta_pu
=\dive(|Du|^{p-2}Du)
=p|Du|^{p-2}\Delta_p^N u 
\]
shows that, setting
\begin{equation}\label{f:def-q}
q:=p-2-\gamma,
\end{equation}
equation
 \eqref{f:eigen-normalized} is equivalent to
\begin{equation}\label{f:weak-equation-formal}
-\Delta_pu
=p\eigeng(\Omega)u^{\gamma+1}|Du|^q\,.
\end{equation}

We next establish the weak formulation
of \eqref{f:eigen-normalized} throughout
the range $-1<\gamma\le p-2$.
The proof relies on the viscosity
results in
\cite[Proposition~2.4 and Section~3.3]
{AttouchiRuosteenoja}
and the viscosity-to-weak implication
in \cite[Theorem~1.4]{MedinaOchoa}.

\begin{proposition}[Weak formulation]\label{p:weak-formulation}
Let $-1 < \gamma \leq p-2$. Then the positive principal eigenfunction $u$ in
\eqref{f:eigen-normalized} 
is a weak solution of \eqref{f:weak-equation-formal}, 
in the sense that $u \in W ^ { 1, p}_{\rm{loc}} (\Omega)$ and 
\begin{equation}\label{f:weak-formulation}
\int_\Omega |Du|^{p-2}Du\cdot D\varphi\,dx
=p\eigeng(\Omega)
\int_\Omega u^{\gamma+1}|Du|^q\varphi\,dx
\end{equation}
for every  $\varphi \in C ^ \infty _c (\Omega)$ (and then by density also for every 
$\varphi\in W^{1,p}(\Omega)\cap L ^ \infty (\Omega)$ with support compactly contained into $\Omega$). 
\end{proposition}

\begin{proof}
Set 
\(
a(x):=p\eigeng(\Omega)u(x)^{\gamma+1}
\).
The continuity and positivity of $u$,
together with its local H\"older regularity,
imply that $a$ is locally H\"older continuous.

By \cite[Theorem~1.1]{AttouchiRuosteenoja},
applied to \eqref{f:eigen-normalized}
with right-hand side $a/p$, we have
$u\in C^{1,\alpha}_{\rm loc}(\Omega)$
for some $\alpha>0$.

As observed in
\cite[Section~3.3]{AttouchiRuosteenoja},
equation \eqref{f:eigen-normalized}
implies
\eqref{f:weak-equation-formal}
in the viscosity sense appropriate
to the $p$-Laplacian.
For $-1<\gamma<0$, this implication
uses the equivalence of viscosity
definitions proved in
\cite[Proposition~2.4]{AttouchiRuosteenoja}.

Fix an open set $U\Subset\Omega$ and
choose $M>\|Du\|_{L^\infty(U)}$.
For $q>0$, define
\[
f_U(x,s,\eta):=
a(x)\min\{|\eta|,M\}^{q},
\]
whereas for $q=0$ set
$f_U(x,s,\eta):=a(x)$.

The function $f_U$ is uniformly
continuous on $U\times\mathbb R
\times\mathbb R^n$, is independent
of $s$, and is bounded. Therefore
it satisfies the assumptions of
\cite[Theorem~1.4]{MedinaOchoa}.

Since $u\in C^1(U)$, the gradient
of any test function touching $u$
at a point $x_0\in U$ equals
$Du(x_0)$. Consequently, the
truncation in the definition of
$f_U$ does not change the viscosity
equation on $U$.

Applying \cite[Theorem~1.4]
{MedinaOchoa} to both subsolutions
and supersolutions, we conclude that
\eqref{f:weak-equation-formal}
holds weakly in $U$.
Since $U\Subset\Omega$ is arbitrary,
the assertion follows.
\end{proof}

As a consequence of the above proposition, we obtain the following
result, which will be  used  to apply the
standard Schwarz rearrangement theory to the positive principal eigenfunction $u$ in
\eqref{f:eigen-normalized}:

\begin{lemma}[Global Sobolev regularity]\label{l:global-W0}
Let $u$ be as in Proposition~\ref{p:weak-formulation}.  Then
\[
u\in W_0^{1,p}(\Omega).
\]
\end{lemma}

\begin{proof}
For $\varepsilon>0$ set $u_\varepsilon:=(u-\varepsilon)_+$.  Since
$u\in C(\overline\Omega)$ and $u=0$ on $\partial\Omega$, the support of
$u_\varepsilon$ is compactly contained in $\Omega$.  Moreover, $u \in W ^ { 1, p}_{\rm{loc}} (\Omega)$. Hence
$u_\varepsilon$ is an admissible test function in
\eqref{f:weak-formulation}.  Using $0\le u\le1$, we obtain
\[
\int_{\{u>\varepsilon\}}|Du|^p\,dx
=p\eigeng(\Omega) \int_{\{u>\varepsilon\}}
u^{\gamma+1}|Du|^q(u-\varepsilon)\,dx
\le p\eigeng(\Omega) \int_{\{u>\varepsilon\}}|Du|^q\,dx.
\]
Since $q<p$, Young's inequality gives, for  a constant $C$ independent of
$\varepsilon$,
\[
\int_{\{u>\varepsilon\}}|Du|^p\,dx\le C |\Omega|.
\]
Letting $\varepsilon\downarrow0$ and using positivity of $u$ in $\Omega$
shows that $u\in W^{1,p}(\Omega)$.  Moreover
$u_\varepsilon\to u$ in $W^{1,p}(\Omega)$, because
$Du_\varepsilon=Du\,\mathbf1_{\{u>\varepsilon\}}$ and
$|Du|^p\in L^1(\Omega)$.  Since each $u_\varepsilon$ has compact support in
$\Omega$, it belongs to $W_0^{1,p}(\Omega)$; by closure, the same holds true for $u$.
\end{proof}

\subsection{On generalized inverses of strictly monotone functions}\label{s:generalized-inverses}

We collect here the facts about generalized inverses which will be used in the
rearrangement argument.  We follow the conventions and basic properties in
\cite{EmbrechtsHofert}; since our distribution function is decreasing rather
than increasing, the statements below are obtained from the corresponding
ones in \cite[Proposition~2.3]{EmbrechtsHofert} by reversing the order.

Let $\mu \colon [0,1] \to [0, L]$ be a strictly decreasing 
right-continuous function,
satisfying 
\[
\mu(0) = L,
\qquad
\lim_{t \to 1-} \mu(t) = 0 = \mu(1).
\]
We define the decreasing generalized inverse of $\mu$ by
\begin{equation}\label{f:def-w}
w(y):=\inf\{t\in[0,1]:\mu(t)\le y\},
\qquad 0\le y\le |\Omega|.
\end{equation}
Thus $w$ is decreasing.  
Right-continuity of $\mu$ gives
\begin{equation}\label{f:geninv-order}
\mu(t)\le y\quad\Longleftrightarrow\quad t\ge w(y),
\end{equation}
while strict monotonicity gives the exact inverse identity
\begin{equation}\label{f:wmu-identity}
w(\mu(t))=t,
\qquad 0\le t\le1.
\end{equation}
In general one only has
\begin{equation}\label{f:mu-w-bracket}
\mu(w(y))\le y\le \mu(w(y)-),
\end{equation}
where $\mu(x-)$ denotes the left limit.  The interval between the two
quantities in \eqref{f:mu-w-bracket} is precisely the interval of values on
which the generalized inverse is constant when $\mu$ has a jump; compare
\cite[Proposition~2.3]{EmbrechtsHofert}.

The following elementary observation will be used repeatedly.

\begin{lemma}[Differentiation of the generalized inverse]\label{l:inverse-differentiation}
The following reciprocal differentiation rules hold.
\begin{enumerate}
\item[(i)] Let $y\in(0,L)$ be a point at which $w$ is differentiable and $w'(y)<0$. 
Then $\mu$ is differentiable at $w(y)$ and
\begin{equation}\label{f:inverse-chain-rule}
\mu'(w(y))\,w'(y)=1.
\end{equation}
\item[(ii)] Let $t\in(0,1)$ be a point at which $\mu$ is differentiable and
$\mu'(t)<0$. Then $y:=\mu(t)$ is the unique point in $(0,L)$ such
that $w(y)=t$, the function $w$ is differentiable at $y$, and
\begin{equation}\label{f:inverse-chain-rule-2}
\mu'(t)w'(y)=1.
\end{equation}
\end{enumerate}
\end{lemma}

\begin{proof}
For \textup{(i)}, set $x=w(y)$. By \eqref{f:mu-w-bracket},
$\mu(x)\le y\le\mu(x-)$. If either inequality were strict, $w$ would be
constant on one side of $y$, contradicting $w'(y)<0$. Hence
$\mu(x)=\mu(x-)=y$, so $\mu$ is continuous at $x$. Since
$w(\mu(t))=t$, letting $t\to x$ gives
\[
w'(y)=\lim_{t\to x}\frac{t-x}{\mu(t)-\mu(x)}.
\]
As $w'(y)\ne0$, taking reciprocals yields \eqref{f:inverse-chain-rule}.

For \textup{(ii)}, set $y=\mu(t)$. 
Continuity of $\mu$ at $t$ and \eqref{f:wmu-identity} give $w(y)=t$, 
while \eqref{f:mu-w-bracket} shows that $w(z)=t$ implies $z=y$. 
If $z\to y$ and $s=w(z)$, monotonicity and continuity imply $s\to t$. 
Writing $a=\mu'(t)<0$, differentiability gives
\[
\mu(s)-y=a(s-t)+o(|s-t|),\qquad
\mu(s-)-y=a(s-t)+o(|s-t|).
\]
Since $\mu(s)\le z\le\mu(s-)$, we obtain
$z-y=a(s-t)+o(|s-t|)$, and therefore
\[
\frac{w(z)-w(y)}{z-y}=\frac{s-t}{z-y}\longrightarrow\frac1a.
\]
This proves \eqref{f:inverse-chain-rule-2}.
\end{proof}

Because $\mu$ need not be absolutely continuous, we also need the following
one-dimensional change-of-variable estimate.

\begin{lemma}[Monotone inverse inequality]\label{l:inverse}
Let $0<m<1$.  Then, for every nonnegative Borel function $\Phi$ and almost
every $y\in(0,L)$,
\begin{equation}\label{f:inverse-integral}
\int_{w(y)}^1
\Phi(t)[-\mu'(t)]^m\,dt
\le
\int_0^y
\Phi(w(x)) [-w'(x)]^{1-m}\,dx.
\end{equation}
Equality sign holds if, in addition, $\mu\in AC_{\rm loc}(0,1)$ or $w\in AC_{\rm loc}(0,L)$.
\end{lemma}

\begin{proof}
Let $\nu=-d\mu$ be the positive Lebesgue--Stieltjes
measure associated with the decreasing, right-continuous
function $\mu$. By the definition of the generalized
inverse, $\nu$ is the push-forward of Lebesgue measure
on $(0,L)$ under $w$. Indeed, for $0<a<b<1$,
\[
\nu((a,b])=\mu(a)-\mu(b)
=\big|\{x\in(0,L):a<w(x)\le b\}\big|.
\]

Set
\[
E=\{t\in(0,1):\mu'(t)<0\},\qquad
\widetilde E=\{x\in(0,L):w'(x)<0\},
\]
where the derivatives are understood at points where
they exist. By Lemma~\ref{l:inverse-differentiation},
for almost every $x\in(0,L)$,
$
x\in\widetilde E$ if and  only if 
$w(x)\in E$; moreover, 
at such points,
\[
-\mu'(w(x))=\frac{1}{-w'(x)}\,,  \qquad x \in \widetilde E \,. 
\]

Define
\[
g(t)=\Phi(t)[-\mu'(t)]^{m-1}\mathbf 1_E(t),
\]
with the convention that $g=0$ outside $E$.
Since $-\mu'(t)\,dt\le\nu$ as positive measures,
we obtain
\[
\begin{aligned}
\int_{w(y)}^1 \Phi(t)[-\mu'(t)]^m\,dt
&=\int_{(w(y),1)}g(t)[-\mu'(t)]\,dt\\
&\le\int_{(w(y),1)}g(t)\,d\nu(t)\\
&=\int_0^y g(w(x))\,dx\\
&=\int_0^y
\Phi(w(x))[-w'(x)]^{1-m}
\mathbf 1_{\widetilde E}(x)\,dx.
\end{aligned}
\]
Here the push-forward identity gives the third line:
if $w$ is constant at the level $w(y)$, the
corresponding contribution vanishes because $g=0$
at jump points of $\mu$. Since the integrand also
vanishes wherever $w'=0$, this proves
\eqref{f:inverse-integral}.

If $\mu\in AC_{\rm loc}(0,1)$, then
$\nu=-\mu'(t)\,dt$ in $(0,1)$, so the preceding
inequality is an equality.

Assume now that $w\in AC_{\rm loc}(0,L)$.
For $0<\varepsilon<y<L$, the one-dimensional
area formula for the nonincreasing function $w$
gives, for every nonnegative Borel function $h$,
\[
\int_{w(y)}^{w(\varepsilon)}h(t)\,dt
=
\int_\varepsilon^y h(w(x))[-w'(x)]\,dx.
\]
We apply this identity with
\[
h(t)=\Phi(t)[-\mu'(t)]^m,
\]
defining $\mu'$ arbitrarily on its set of
non-differentiability. At points where $w'<0$,
Lemma~\ref{l:inverse-differentiation} yields
\[
h(w(x))[-w'(x)]
=
\Phi(w(x))[-w'(x)]^{1-m},
\]
while both sides vanish when $w'=0$.
Consequently,
\[
\int_{w(y)}^{w(\varepsilon)}
\Phi(t)[-\mu'(t)]^m\,dt
=
\int_\varepsilon^y
\Phi(w(x))[-w'(x)]^{1-m}\,dx.
\]
Letting $\varepsilon\downarrow0$ and using
$w(\varepsilon)\uparrow1$, monotone convergence
gives equality in \eqref{f:inverse-integral}.
\end{proof}

\section{Reduction to a radial inequality for the Schwarz rearrangement}\label{s:1d}

Let $X_\delta$ be the completion of
\[
\mathcal D_R:=
\{\phi\in C^1([0,R]):\phi(R)=0\}
\]
with respect to the norm
\[
\|\phi\|_{X_\delta}
:=
\left(\int_0^R
r^\delta|\phi'|^{\gamma+2}\,dr\right)^{\frac{1}{\gamma+2}},
\]
where $\delta := (n-1)(\gamma+1)/(p-1)$ is the
exponent defined in \eqref{f:var}.

\begin{lemma}[Weighted radial problem]
\label{l:weighted-radial}
Let $z$ be the positive radial principal eigenfunction
in $B_R$, normalized by $z(0)=1$. Set
\[
\Lambda:=p'(\gamma+1)\lambda_{p,\gamma}(B_R).
\]
Then $z\in C^1((0,R))$, $z'<0$ in $(0,R)$, and
\begin{equation}\label{f:ball-integral}
r^\delta[-z'(r)]^{\gamma+1}
=
\Lambda\int_0^r
\rho^\delta z(\rho)^{\gamma+1}\,d\rho,
\qquad 0<r<R.
\end{equation}
Equivalently,
\begin{equation}\label{f:weighted-ode}
-\bigl(r^\delta|z'|^\gamma z'\bigr)'
=
\Lambda r^\delta z^{\gamma+1}
\qquad\text{in }(0,R),
\end{equation}
with $z(R)=0$ and the natural condition
$\lim_{r\downarrow0}r^\delta|z'(r)|^{\gamma+1}=0$. 
Moreover, 
\begin{equation}\label{f:weighted-Rayleigh}
\Lambda
=
\min_{\phi\in X_\delta\setminus\{0\}}
\frac{\displaystyle\int_0^R
r^\delta|\phi'|^{\gamma+2}\,dr}
{\displaystyle\int_0^R
r^\delta|\phi|^{\gamma+2}\,dr}\,,
\end{equation}
$z$ is a minimizer, and every positive
minimizer is a positive multiple of $z$.
\end{lemma}

\begin{proof}
Set $\lambda:=\lambda_{p,\gamma}(B_R)$
and recall that $q:=p-2-\gamma\geq 0$.

\emph{Step 1: Radial flux and strict monotonicity.}
By Proposition~2, the eigenfunction is a weak
solution of
\[
-\Delta_p z=p\lambda z^{\gamma+1}|Dz|^q
\qquad\text{in }B_R.
\]
Consequently, its radial profile satisfies
\begin{equation}\label{f:radial-flux-eq}
H'(r)
=
p\lambda r^{n-1}z(r)^{\gamma+1}|z'(r)|^q
\qquad\text{in }\mathcal D'(0,R),
\end{equation}
where
\[
H(r):= -r^{n-1}|z'(r)|^{p-2}z'(r).
\]

We first establish the zero-flux condition
at the origin. Since $z$ is bounded,
$z\in W^{1,p}_{\mathrm{loc}}(B_R)$, and $0\le q<p$,
the right-hand side of
\eqref{f:radial-flux-eq} belongs to $L^1(0,b)$
for every $b<R$. Hence $H$ has an absolutely
continuous representative on $[0,b]$, and
$H(0+)$ exists.

Choose a smooth radial test function
$\varphi(x)=\eta(|x|)$, compactly supported
in $B_R$ and constant near the origin.
The weak formulation, after cancellation
of the factor $|\mathbb S^{n-1}|$, gives
\[
-\int_0^R H(r)\eta'(r)\,dr
=
p\lambda\int_0^R
r^{n-1}z^{\gamma+1}|z'|^q\eta(r)\,dr.
\]
Integrating the left-hand side by parts
and using \eqref{f:radial-flux-eq}, we obtain
\[
H(0+)\eta(0)+\int_0^R H'(r)\eta(r)\,dr
=
\int_0^R H'(r)\eta(r)\,dr.
\]
Since $\eta(0)$ is arbitrary, $H(0+)=0$.
Thus
\begin{equation}\label{f:radial-p-flux}
H(r)
=
p\lambda\int_0^r
\rho^{n-1}z(\rho)^{\gamma+1}
|z'(\rho)|^q\,d\rho.
\end{equation}

In particular, $H\geq 0$ and $z$ is nonincreasing. 
We claim that $H(r) > 0$ for every $r\in(0,R)$. 
Otherwise, if $H(r_0)=0$ for some $r_0>0$, the
monotonicity of $H$ and $H(0)=0$ imply $H=0$ on $(0,r_0)$. 
Hence $z$ is constant on $B_{r_0}$. 
This contradicts the viscosity supersolution condition for
\[
-|Dz|^\gamma\Delta_p^N z
=\lambda z^{\gamma+1}>0.
\]
For $\gamma\ge0$ the contradiction follows
by testing with the constant function.
For $-1<\gamma<0$ it follows directly
from the locally constant case in the
Birindelli--Demengel viscosity definition
adopted in Section~\ref{s:weak}.

Therefore $H>0$ on $(0,R)$. Since $H$ is continuous and
\[
z'(r)
=
-\left(\frac{H(r)}{r^{n-1}}\right)^{1/(p-1)}
\quad\text{for a.e. }r\in(0,R),
\]
the absolute continuity of $z$ on compact
subintervals shows that $z\in C^1((0,R))$
and $z'<0$ throughout $(0,R)$.

\emph{Step 2: Weighted equation and Rayleigh
characterization.}
By \eqref{f:radial-p-flux}, $H(0)=0$ and
\[
H'(r)
=
p\lambda r^{n-1}z^{\gamma+1}[-z'(r)]^q
\quad\text{a.e. in }(0,R).
\]
By the definitions of $\delta$, $q$, and $\Lambda$,
we obtain on $(0,R)$
\[
\left(H^{\frac{\gamma+1}{p-1}}\right)'
=
\frac{\gamma+1}{p-1}\, H^{-\frac{q}{p-1}} H'
=
p \lambda\, \frac{\gamma+1}{p-1}\, r^\delta z^{\gamma+1}
=
\Lambda\, \, r^\delta z^{\gamma+1}.
\]
As $H^{\frac{\gamma+1}{p-1}}= r^\delta(-z')^{\gamma+1}$ and
$H(0)=0$, integration yields \eqref{f:ball-integral}. 
Equation \eqref{f:weighted-ode} and its zero-flux
condition follow immediately.

We next justify the variational formulation.
For every $\phi\in\mathcal D_R$,
integration by parts and Holder's
inequality give
\[
\begin{aligned}
\int_0^R r^\delta|\phi|^{\gamma+2}\,dr
&=
-\frac{{\gamma+2}}{\delta+1}
\int_0^R r^{\delta+1}
|\phi|^{\gamma}\phi\phi'\,dr
\leq
\frac{(\gamma+2)R}{\delta+1}
\int_0^R r^{\delta}
|\phi|^{\gamma+1}|\phi'|\,dr
\\ & \leq
\frac{({\gamma+2})R}{\delta+1}
\left(\int_0^R r^\delta|\phi|^{\gamma+2}\,dr
\right)^{\frac{\gamma+1}{\gamma+2}}
\left(\int_0^R r^\delta|\phi'|^{\gamma+2}\,dr
\right)^{\frac{1}{\gamma+2}}.
\end{aligned}
\]
Consequently,
\begin{equation}\label{f:weighted-Poincare}
\|\phi\|_{L^{\gamma+2}(r^\delta dr)}
\le
\frac{(\gamma+2)R}{\delta+1}
\|\phi\|_{X_\delta}.
\end{equation}
This proves that the weighted gradient norm
is a norm on $\mathcal D_R$ and that
$X_\delta$ embeds continuously into
$L^{\gamma+2}((0,R),r^\delta\,dr)$.

Since $z$ is bounded, \eqref{f:ball-integral}
gives
\[
-z'(r)=O(r^{1/(\gamma+1)})
\qquad\text{as }r\downarrow0.
\]
Hence $z'$ extends continuously to $r=0$
with $z'(0)=0$. Moreover, the integral
on the right-hand side of
\eqref{f:ball-integral} has a finite limit
as $r\uparrow R$. Since $R>0$, the same
identity shows that $z'$ has a finite
limit as $r\uparrow R$. Thus
$z\in C^1([0,R])$, and, since $z(R)=0$,
we have $z\in\mathcal D_R\subset X_\delta$.

We now use Picone's inequality (see e.g.~\cite{AllegrettoHuang1998}). For
$\phi\in\mathcal D_R$ and $\varepsilon>0$,
\[
|\phi'|^{\gamma+2}
\ge
|z'|^{\gamma}z'
\left(
\frac{|\phi|^{\gamma+2}}
{(z+\varepsilon)^{\gamma+1}}
\right)'
\qquad\text{a.e. in }(0,R).
\]
Multiplying by $r^\delta$ and integrating by parts,
the boundary terms vanish at $r=0$ by
the zero-flux condition, and at $r=R$
because $\phi(R)=0$. 
Using \eqref{f:weighted-ode}, we obtain
\[
\int_0^R r^\delta|\phi'|^{\gamma+2}\,dr
\ge
\Lambda\int_0^R r^\delta z^{\gamma+1}
\frac{|\phi|^{\gamma+2}}
{(z+\varepsilon)^{\gamma+1}}\,dr.
\]
Letting $\varepsilon\downarrow0$ and
using monotone convergence gives
\begin{equation}\label{f:weighted-Picone-bound}
\int_0^R r^\delta|\phi'|^{\gamma+2}\,dr
\ge
\Lambda\int_0^R r^\delta|\phi|^{\gamma+2}\,dr.
\end{equation}
By density and \eqref{f:weighted-Poincare},
this inequality holds for every
$\phi\in X_\delta$.

On the other hand, testing
\eqref{f:weighted-ode} with $z$ gives
\[
\int_0^R r^\delta|z'|^{\gamma+2}\,dr
=
\Lambda\int_0^R r^\delta z^{\gamma+2}\,dr.
\]
Thus $z$ attains the lower bound in
\eqref{f:weighted-Picone-bound}, proving
\eqref{f:weighted-Rayleigh}.

\emph{Step 3: Simplicity of the positive
minimizer.}
Let $\phi\in X_\delta$ be another positive
minimizer. Then
\begin{equation}\label{f:weighted-equality}
\int_0^R r^\delta|\phi'|^{\gamma+2}\,dr
=
\Lambda\int_0^R r^\delta\phi^{\gamma+2}\,dr.
\end{equation}
Every element of $X_\delta$ belongs to
$W^{1,\gamma+2}_{\mathrm{loc}}(0,R)$, so that
the pointwise Picone inequality applies
to $\phi$ and $z$ almost everywhere.

To justify the integrated equality case, fix
$\varepsilon>0$ and, for $M>0$, set
$\phi_M:=\min\{\phi,M\}$. Define
\[
\psi_{M,\varepsilon}
:=
\frac{(\phi_M)^{\gamma+2}}{(z+\varepsilon)^{\gamma+1}}.
\]
Since $\phi_M\in X_\delta$ is bounded,
$z\in C^1([0,R])$, and
$z+\varepsilon\ge\varepsilon$, the
Lipschitz chain rule and multiplication
by a bounded Lipschitz function show
that $\psi_{M,\varepsilon}\in X_\delta$.

It is therefore an admissible test function
in the weak formulation of
\eqref{f:weighted-ode}, and
\[
\int_0^R r^\delta|z'|^{\gamma}z'
\psi_{M,\varepsilon}'\,dr
=
\Lambda\int_0^R r^\delta z^{\gamma+1}
\psi_{M,\varepsilon}\,dr.
\]
Picone's inequality applied to $\phi_M$
and $z+\varepsilon$ gives
\[
0\leq
\int_0^R r^\delta
\left[
|\phi_M'|^{\gamma+2}
-
|z'|^{\gamma}z'
\psi_{M,\varepsilon}'
\right]dr.
\]

We now let $M\uparrow\infty$.
For fixed $\varepsilon>0$, the chain rule
gives
\[
|\psi_{M,\varepsilon}'|
\le
C_\varepsilon
\left(
\phi^{\gamma+1}|\phi'|
+
\phi^{\gamma+2}|z'|
\right)
\qquad\text{a.e. in }(0,R),
\]
where $C_\varepsilon$ is independent of $M$.
The right-hand side belongs to
$L^1((0,R),r^\delta\,dr)$.
Indeed,
\[
\int_0^R r^\delta
\phi^{\gamma+1}|\phi'|\,dr
\le
\left(\int_0^R r^\delta\phi^{\gamma+2}\,dr
\right)^{\frac{\gamma+1}{\gamma+2}}
\left(\int_0^R r^\delta|\phi'|^{\gamma+2}\,dr
\right)^{\frac{1}{\gamma+2}},
\]
while $\phi^{\gamma+2}|z'|$ is integrable because
$z'$ is bounded.

Moreover, $\phi_M\to\phi$ and
$\phi_M'\to\phi'$ almost everywhere,
with $|\phi_M'|\le|\phi'|$.
Dominated convergence therefore yields
\[
\begin{split}
0 & \leq
\int_0^R r^\delta
\biggl[
|\phi'|^{\gamma+2}
-
|z'|^{\gamma}z'
\left(
\frac{\phi^{\gamma+2}}{(z+\varepsilon)^{\gamma+1}}
\right)'
\biggr]\,dr
\\ & =
\int_0^R r^\delta|\phi'|^{\gamma+2}\,dr
-
\Lambda\int_0^R r^\delta
z^{\gamma+1}\frac{\phi^{\gamma+2}}
{(z+\varepsilon)^{\gamma+1}}\,dr.
\end{split}
\]

By \eqref{f:weighted-equality}, the last
expression tends to zero as
$\varepsilon\downarrow0$. Since the
Picone integrand is nonnegative,
Fatou's lemma implies that
\[
|\phi'|^{\gamma+2}
=
|z'|^{\gamma}z'
\left(\frac{\phi^{\gamma+2}}{z^{\gamma+1}}\right)'
\qquad\text{a.e. in }(0,R).
\]
The equality case in Picone's inequality
therefore gives
\[
\left(\frac{\phi}{z}\right)'=0
\qquad\text{a.e. in }(0,R).
\]
Thus $\phi=cz$ for some $c>0$,
proving uniqueness of the positive
minimizer up to multiplication by a
constant.
\end{proof}

\bigskip
Write $|\Omega|=\omega_nR^n$, let $u^*$ denote the Schwarz rearrangement of $u$ on $B_R$, and let $v$ be defined on $[0, R]$ by 
\[
u^*(x)=v(|x|).
\]
Namely,  $v$ is the radial profile of the Schwarz rearrangement of $u$ (in particular, $v$ is nonincreasing, $v(0)=1$, and $v(R)=0$).

\begin{proposition}[Reduction to the radial inequality]\label{p:1d-comparison}
Assume that the rearrangement $v$ satisfies 
\begin{equation}\label{f:key-gamma}
r^\delta[-v'(r)]^{\gamma+1}
\le
p' (\gamma+1) \eigeng(\Omega)
\int_0^r\rho^\delta v(\rho)^{\gamma+1}\,d\rho \qquad \text{ for a.e. } r \in (0, R) \,. 
\end{equation}
 Then
\begin{equation}\label{f:lambda-comparison}
\eigeng(\Omega)\ge\eigeng(B_R).
\end{equation}
If equality holds, then $v=z$ after the normalization $v(0)=z(0)=1$, and
\eqref{f:key-gamma} is an equality for almost every $r\in(0,R)$.
\end{proposition}

\begin{proof}
Multiply \eqref{f:key-gamma} by $-v'(r)\ge0$ and integrate over $(0,R)$.
With
\[
G(r):=\int_0^r\rho^\delta v(\rho)^{\gamma+1}\,d\rho
\]
we obtain
\[
\int_0^Rr^\delta|v'|^{\gamma+2}\,dr
\le
p' (\gamma+1) \eigeng(\Omega)
\int_0^R[-v'(r)]G(r)\,dr.
\]
An integration by parts, using $v(R)=0$ and $G(0)=0$, gives
\[
\int_0^R[-v']G
=\int_0^Rr^\delta v^{\gamma+2}\,dr.
\]
Therefore
\begin{equation}\label{f:Rayleigh-v}
\frac{\displaystyle\int_0^Rr^\delta|v'|^{\gamma+2}\,dr}
{\displaystyle\int_0^Rr^\delta v^{\gamma+2}\,dr}
\le p' (\gamma+1) \eigeng(\Omega).
\end{equation}

We claim that $v\in X_\delta$. Indeed, since
$0\le v\le1$, the function
\[
G(r)=\int_0^r\rho^\delta v(\rho)^{\gamma+1}\,d\rho
\]
is bounded on $[0,R]$. Moreover, $v$ is
nonincreasing and belongs to
$AC_{\rm loc}(0,R)$. Hence
\[
\int_0^R[-v'(r)]G(r)\,dr
\le G(R)\int_0^R[-v'(r)]\,dr
\le G(R)<\infty.
\]
The preceding inequality therefore yields
\[
\int_0^Rr^\delta|v'|^{\gamma+2}\,dr<\infty.
\]

To conclude, recall that $v(R)=0$.
For $\varepsilon>0$, define
\[
v_\varepsilon(r):=
\begin{cases}
v(\varepsilon),&0\le r\le\varepsilon,\\
v(r),&\varepsilon<r\le R.
\end{cases}
\]
Then $v_\varepsilon$ is absolutely continuous
on $[0,R]$, vanishes at $R$, and
\[
\int_0^R r^\delta
|v_\varepsilon'-v'|^{\gamma+2}\,dr
=
\int_0^\varepsilon
r^\delta|v'|^{\gamma+2}\,dr
\longrightarrow0.
\]
Since the weight $r^\delta$ is bounded
above and below by positive constants
on $[\varepsilon,R]$, each
$v_\varepsilon$ can be approximated
in the weighted gradient norm by
functions in $\mathcal D_R$.
It follows that $v\in X_\delta$.

By \eqref{f:weighted-Rayleigh}, the left-hand side is at least
$p' (\gamma+1) \eigeng(B_R)$, proving \eqref{f:lambda-comparison}.

Assume now equality of the eigenvalues.  Then equality holds in both sides
of the Rayleigh comparison, so $v$ is a positive first minimizer in
\eqref{f:weighted-Rayleigh}.  After normalization, $v=z$.  Since
$z'(r)<0$ for every $r\in(0,R)$, the nonnegative defect in
\eqref{f:key-gamma}, multiplied by $-v'$, has zero integral.  Hence that
defect vanishes almost everywhere, and \eqref{f:key-gamma} is an equality
for almost every $r$.
\end{proof}

Consequently,  the proof of the Faber--Krahn inequality
is reduced to establishing \eqref{f:key-gamma} for the radial profile $v$.
The next section is entirely devoted to this task.

\section{Proof of the radial differential inequality  \eqref{f:key-gamma}}\label{s:proof-radial}

The proof requires some preliminary lemmas. For $t\in(0,1)$, let us set
\begin{equation}\label{f:levels-u}
U_t:=\{u>t\},
\qquad \text{ and } \qquad 
S_t:=\{u=t\}
\end{equation}
and define 
\begin{equation}\label{f:def-F}
F(t):=\int_{S_t}|Du|^{p-1}\,d\mathcal H^{n-1}.
\end{equation}

\begin{lemma}[Flux identity]\label{l:flux}
For almost every $t\in(0,1)$,
\begin{equation}\label{f:flux-identity}
F(t)
=p\eigeng(\Omega)
\int_{U_t}u^{\gamma+1}|Du|^q\,dx.
\end{equation}
Moreover, if $q>0$, we henceforth identify $F$
with the representative defined for every
$t\in(0,1)$ by the right-hand side of
\eqref{f:flux-identity}. 
Then 
$F\in AC([0,1])$, and
\begin{equation}\label{f:Fprime}
-F'(t)
=p\eigeng(\Omega)t^{\gamma+1}
\int_{S_t\cap\{|Du|>0\}}
|Du|^{q-1}\,d\mathcal H^{n-1}
\qquad\text{for a.e. }t\in(0,1).
\end{equation}
Furthermore,
\begin{equation}\label{f:F-to-zero}
F(t)\longrightarrow0
\qquad\text{as }t\to1^-.
\end{equation}
\end{lemma}

\begin{proof}
Fix $t\in(0,1)$ and $\e>0$ sufficiently small.
Define
\[
\varphi_{t,\e}(x):=
\begin{cases}
0, & u(x)\le t,\\[1mm]
\dfrac{u(x)-t}{\e}, & t<u(x)<t+\e,\\[2mm]
1, & u(x)\ge t+\e.
\end{cases}
\]
Since $t>0$, the support of $\varphi_{t,\e}$
is compactly contained in $\Omega$. Moreover,
\[
D\varphi_{t,\e}
=\frac1\e Du\,
\mathbf 1_{\{t<u<t+\e\}}
\qquad\text{a.e. in }\Omega.
\]
Using $\varphi_{t,\e}$ as a test function in
\eqref{f:weak-formulation}, we obtain
\begin{equation}\label{f:truncation-identity}
\frac1\e
\int_{\{t<u<t+\e\}}|Du|^p\,dx
=
p\eigeng(\Omega)
\int_\Omega
u^{\gamma+1}|Du|^q
\varphi_{t,\e}\,dx.
\end{equation}

By the coarea formula,
\[
\begin{aligned}
\frac1\e
\int_{\{t<u<t+\e\}}|Du|^p\,dx
&=
\frac1\e\int_t^{t+\e}
\left(
\int_{S_\tau}|Du|^{p-1}\,
d\mathcal H^{n-1}
\right)d\tau\\
&=
\frac1\e\int_t^{t+\e}F(\tau)\,d\tau.
\end{aligned}
\]
Since $u$ is bounded,
$u\in W^{1,p}_0(\Omega)$, and $0\le q<p$,
the function $u^{\gamma+1}|Du|^q$
belongs to $L^1(\Omega)$.
Moreover,
$\varphi_{t,\e}\to\mathbf 1_{U_t}$
pointwise as $\e\downarrow0$.
Dominated convergence on the right-hand
side of \eqref{f:truncation-identity}
and Lebesgue differentiation on the
left-hand side yield
\eqref{f:flux-identity} for almost every
$t\in(0,1)$.

Assume now that $q>0$.
In this case,
$u^{\gamma+1}|Du|^q$ vanishes almost
everywhere on the critical set
$\{|Du|=0\}$.
Applying the coarea formula with
\[
g(x):=
u(x)^{\gamma+1}|Du(x)|^{q-1}
\mathbf 1_{\{|Du|>0\}},
\]
we obtain
\begin{equation}\label{f:flux-coarea}
F(t)
=
p\eigeng(\Omega)
\int_t^1\tau^{\gamma+1}
\left(
\int_{S_\tau\cap\{|Du|>0\}}
|Du|^{q-1}\,d\mathcal H^{n-1}
\right)d\tau
\end{equation}
for every $t\in(0,1)$, where $F$
is now understood as the representative
specified in the statement.

The integrand on the right-hand side
is nonnegative and belongs to
$L^1(0,1)$, since
\[
\begin{aligned}
&\int_0^1\tau^{\gamma+1}
\left(
\int_{S_\tau\cap\{|Du|>0\}}
|Du|^{q-1}\,d\mathcal H^{n-1}
\right)d\tau\\
&\qquad=
\int_\Omega u^{\gamma+1}|Du|^q\,dx
<\infty.
\end{aligned}
\]
Consequently, \eqref{f:flux-coarea} implies 
$F\in AC([0,1])$,
and differentiation gives
\eqref{f:Fprime} for almost every
$t\in(0,1)$.

Finally, since the integrand in
\eqref{f:flux-coarea} belongs to
$L^1(0,1)$, its integral over $(t,1)$
tends to zero as $t\to1^-$.
This proves \eqref{f:F-to-zero}.
\end{proof}

\bigskip

Next we state a geometric lower bound  for  the flux. 

 For any $t \in (0, 1)$, we denote by $P ( U_t)$  the 
perimeter of $U_t$, and we set 
\begin{equation}
\mu(t):=|U_t|\,.
\end{equation} 
 We extend $\mu$ to $[0,1]$ by setting $\mu(0)=|\Omega|$ and $\mu(1)=0$.
Notice that  $\mu$ is strictly
decreasing on $[0,1]$:  if $0<s<t<1$, continuity of $u$, positivity
in $\Omega$, the normalization $\max_{\overline\Omega}u=1$, and connectedness
of $\Omega$ imply that the open set $\{s<u<t\}$ is nonempty; hence it has
positive measure and $\mu(s)>\mu(t)$.  The endpoint cases follow in the same way from 
positivity in $\Omega$ and the zero boundary condition. Moreover,
$\mu$ is right-continuous: if $t_j\downarrow t$, then
$\{u>t_j\}\uparrow\{u>t\}$, and continuity from below of Lebesgue measure
gives $\mu(t_j)\to\mu(t)$.

We also observe for later use that, since  by Lemma~\ref{l:global-W0} and the
P\'olya--Szeg\H{o} inequality, $u^*$ belongs to $W^{1,p}_0(B_R)$, the  radial
profile function $v$ 
belongs to $W^{1,p}_{\mathrm{loc}}((0,R))\subset AC_{\mathrm{loc}}((0,R))$, 
and hence the function 
$w(y):=v \big (\big (\frac{y}{\omega_n}\big)^{1/n}\big)$
belongs to $AC_{\mathrm{loc}}((0,|\Omega|))$. 
Since $w$ is the generalized inverse of $\mu$,  Lemma~\ref{l:inverse-differentiation} shows that, 
 at every
point where $v$ is differentiable and $v'(r)<0$,
$\mu$ is differentiable at $v(r)$ and
\begin{equation}\label{f:distribution-v}
\mu(v(r))=\omega_nr^n,
\qquad
[-\mu'(v(r))][-v'(r)]=n\omega_nr^{n-1}.
\end{equation}

We shall also use the following consequence of the
one-dimensional area formula. If
$N\subset(0,1)$ has Lebesgue measure zero,
then, since $v\in AC_{\rm loc}(0,R)$,
\[
\int_{\{r\in(0,R):\,v(r)\in N\}}
|v'(r)|\,dr=0.
\]
Consequently,
\[
\bigl|\{r\in(0,R):\,v(r)\in N,\ v'(r)<0\}\bigr|=0.
\]
Thus any property holding for almost every
$t\in(0,1)$ may be evaluated at $t=v(r)$
for almost every $r$ such that $v'(r)<0$.

We also record for later use that, from Lemma~\ref{l:inverse},
choosing $r\in (0,R)$ and $y = \omega_n r^n$, we obtain that,
for every nonnegative Borel function $\Phi$,
\begin{equation}\label{f:inverse}
\int_{v(r)}^1
\Phi(t)[-\mu'(t)]^m\,dt
=
(n\omega_n)^m
\int_0^r
\Phi(v(\rho))\rho^{m(n-1)}[-v'(\rho)]^{1-m}\,d\rho.
\end{equation}

Below, we set 
\[  Y(t):=F(t)^m\,, \qquad  m:=\frac{\gamma+2}{p} \,. \]

 \smallskip 
 
 \begin{lemma}[Geometric lower bound for the flux]\label{l:perimeter-flux}
For almost every $t\in(0,1)$,
\begin{equation}\label{f:perimeter-holder}
P(U_t)^p
\le F(t)[-\mu'(t)]^{p-1}.
\end{equation}
Consequently,
\begin{equation}\label{f:Y-perimeter}
Y(t)[-\mu'(t)]^{m(p-1)}
\ge P(U_t)^{\gamma+2}
\ge
n^{\gamma+2}\omega_n^{(\gamma+2)/n}
\mu(t)^{(n-1)(\gamma+2)/n}.
\end{equation}
\end{lemma}
 
 \begin{proof}
For almost every $t$, the reduced boundary of $U_t$ is represented, up to
$\mathcal H^{n-1}$-null sets, by the corresponding level surface on the
noncritical part.  H\"older's inequality yields
\[
P(U_t)
\le
\int_{S_t\cap\{|Du|>0\}}1\,d\mathcal H^{n-1}
\le
F(t)^{1/p}Q(t)^{(p-1)/p}\,,
\]
where
\begin{equation}\label{f:def-Q}
Q(t):=
\int_{S_t\cap\{|Du|>0\}}\frac1{|Du|}\,d\mathcal H^{n-1}.
\end{equation}

The decreasing function $\mu$ satisfies, in the sense of positive Radon
measures on $(0,1)$,
\begin{equation}\label{f:measure-Q-mu}
Q(t)\,dt\le -d\mu(t).
\end{equation}
Consequently,
\begin{equation}\label{f:Q-muprime}
Q(t)\le -\mu'(t)
\end{equation}
for almost every $t$ at which the ordinary derivative of the monotone
function $\mu$ exists.

Indeed, for $0<a<b<1$, the coarea formula gives
\[
\int_a^bQ(t)\,dt
=|\{a<u<b,\ |Du|>0\}|
\le |\{a<u<b\}|
\le \mu(a)-\mu(b^-).
\]
This is exactly \eqref{f:measure-Q-mu}.  Taking Radon--Nikodym densities of
the absolutely continuous parts yields \eqref{f:Q-muprime}.
 Using  \eqref{f:Q-muprime}  gives \eqref{f:perimeter-holder}.  Raising it to the
power $m=(\gamma+2)/p$ gives the first inequality in
\eqref{f:Y-perimeter}.  The second one is the Euclidean isoperimetric
inequality for sets of finite perimeter; see, for instance,
\cite[Chapter~12]{Maggi}.
\end{proof}

\bigskip 
We are now ready to prove the  
target radial differential inequality.  
We proceed separately in the strict range case $\gamma < p-2$ and the endpoint case $\gamma = p -2$.

\bigskip 
{\bf Proof of   \eqref{f:key-gamma} for  $\gamma < p-2$}.  
We recall that 
\[m := \frac{\gamma+2}{p} \in (0,1),
\quad
q := p - 2 - \gamma \in (0, p-1)\,,
\quad
\delta:=\frac{(n-1)(\gamma+1)}{p-1}  \in (0, n-1)\,.
\] 
We argue in three steps. 

\bigskip
\underbar {Step 1} (differential estimate for the flux).

\smallskip
We claim that 
\begin{equation}\label{f:Yt-differential}
-Y'(t)
\le
(\gamma+2)\eigeng(\Omega)
t^{\gamma+1}[-\mu'(t)]^m \qquad \text{ for a.e. } t \in (0, 1)
\end{equation}
and, consequently,
\begin{equation}\label{f:Yt-integrated}
Y(t)
\le
(\gamma+2)\eigeng(\Omega)
\int_t^1 \tau^{\gamma+1}[-\mu'(\tau)]^m\,d\tau.
\end{equation}

Indeed, since
$q-1=(1-m)(p-1)+m(-1)$,
H\"older's inequality on $S_t\cap\{|Du|>0\}$ gives
\begin{equation}\label{f:Holder-J}
\int_{S_t\cap\{|Du|>0\}}|Du|^{q-1}\,d\mathcal H^{n-1}
\le
F(t)^{1-m} Q(t)^m.
\end{equation}
Starting from the equality \eqref{f:Fprime} obtained in Lemma \ref{l:flux},  and using first \eqref{f:Holder-J} and then
\eqref{f:Q-muprime}, we obtain 
\[
-F'(t)
\le
p\eigeng(\Omega)t^{\gamma+1}
F(t)^{1-m}[-\mu'(t)]^m.
\]
Hence,
\[
-Y'(t)
=mF(t)^{m-1}[-F'(t)]
\le
mp\eigeng(\Omega)t^{\gamma+1}[-\mu'(t)]^m,
\]
 which proves \eqref{f:Yt-differential} recalling that $mp=\gamma+2$.  Integrating from
$t$ to $1$ and using condition \eqref{f:F-to-zero} in Lemma \ref{l:flux} gives
\eqref{f:Yt-integrated}.

\bigskip

\medskip 
\underbar{Step 2} (combining geometric lower bound and differential estimate for the flux).

\smallskip We claim that, for a.e.\ $r\in (0,R)$,
\begin{equation}\label{f:intermediate-radial}
r^{m(n-1)}[-v'(r)]^{m(p-1)}
\leq
(\gamma+2)\eigeng(\Omega)
\int_0^r
\rho^{m(n-1)} v(\rho)^{\gamma+1}[-v'(\rho)]^{1-m}\,d\rho .
\end{equation}

By the preceding area-formula argument,
\eqref{f:Y-perimeter} holds at $t=v(r)$
for almost every $r$ such that $v'(r)<0$.
Using \eqref{f:distribution-v}, we obtain 
\begin{equation}
\begin{split}
Y(v(r))
&\geq
\frac{(n\omega_nr^{n-1})^{\gamma+2}}
{[-\mu'(v(r))]^{m(p-1)}}
\\
&\geq
(n\omega_n)^m r^{m(n-1)}
[-v'(r)]^{m(p-1)},
\label{f:Y-lower-r}
\end{split}
\end{equation}
where we have used the relations 
$\gamma + 2 = mp$.

On the other hand, the inequality \eqref{f:Yt-integrated} obtained in Step 1 and 
Lemma~\ref{l:inverse}, applied with
$\Phi(t)=t^{\gamma+1}$,
give
\begin{equation}\label{f:Y-upper-r}
Y(v(r))
\leq
(\gamma+2)\eigeng(\Omega)(n\omega_n)^m
\int_0^r
\rho^{m(n-1)}v(\rho)^{\gamma+1}[-v'(\rho)]^{1-m}
\,d\rho.
\end{equation}
Comparison of \eqref{f:Y-lower-r} and \eqref{f:Y-upper-r} yields
\eqref{f:intermediate-radial} at almost every point where $v'<0$.  
At points where $v'=0$ its left-hand side is zero, so the same inequality holds
there as well.

\medskip 
\underbar{Step 3} (conclusion). 

\smallskip 
The inequality  \eqref{f:intermediate-radial} obtained in Step~2 can be rewritten as
\begin{equation}\label{f:I-ineq}
r^{m(n-1)}[-v'(r)]^{m(p-1)}\leq (\gamma+2)\eigeng(\Omega) I(r),
\end{equation}
where
\[
I(r):=\int_0^r\rho^{m(n-1)} v(\rho)^{\gamma+1}[-v'(\rho)]^{1-m}\,d\rho.
\]

Raising \eqref{f:I-ineq} to the power 
$\frac{\delta}{m(n-1)}$
and observing that $\frac{(p-1)\delta}{n-1} = \gamma + 1$
gives
\begin{equation}\label{f:first-K}
r^\delta[-v'(r)]^{\gamma+1}\leq 
[(\gamma+2)\eigeng(\Omega)]^{\frac{\delta}{m(n-1)}} J(r),
\qquad
J(r) := I(r)^{\frac{\delta}{m(n-1)}}\,.
\end{equation}
Furthermore,
\[
\begin{split}
J'(r) & = {\frac{\delta}{m(n-1)}} I(r)^{{\frac{\delta}{m(n-1)}}-1} I'(r)
\\ & =
{\frac{\delta}{m(n-1)}} I(r)^{{\frac{\delta}{m(n-1)}}-1}
\, r^{m(n-1)} v(r)^{\gamma+1}[-v'(r)]^{1-m}\,.
\end{split}
\]
Using \eqref{f:first-K},
\[
[-v'(r)]^{1-m} \leq
[(\gamma+2)\eigeng(\Omega)]^{\frac{1-m}{m(p-1)}} 
r^{-\frac{(1-m)(n-1)}{p-1}}
I(r)^{\frac{1-m}{m(p-1)}}\,.
\]
Hence, observing that
\begin{equation}
\label{f:idexp}
\frac{\delta}{m(n-1)}+\frac{1-m}{m(p-1)} = 1,
\end{equation}
we obtain
\[
J'(r) \leq 
\frac{\delta}{m(n-1)}
[(\gamma+2)\eigeng(\Omega)]^{\frac{1-m}{m(p-1)}} 
r^\delta v(r)^{\gamma+1}\,.
\]
As $J(0)=0$,
\begin{equation}
\label{f:second-K}
J(r)
\leq
\frac{\delta}{m(n-1)}
[(\gamma+2)\eigeng(\Omega)]^{\frac{1-m}{m(p-1)}} 
\int_0^r\rho^\delta v(\rho)^{\gamma+1}\,d\rho.
\end{equation}
Combining this with \eqref{f:first-K}, 
recalling \eqref{f:idexp}
and $\frac{\delta(\gamma+2)}{m(n-1)} = p' (\gamma + 1)$,
we finally obtain \eqref{f:key-gamma}.
\qed

\bigskip 
\medskip 

{\bf Proof of   \eqref{f:key-gamma} for  $\gamma = p-2$}.  

In this case, there is no  need of a differential estimate for the flux such as the one obtained in Step 1 above. 
Indeed,  we can proceed more directly as follows. 
By Lemma \ref{l:perimeter-flux}
and the Euclidean
isoperimetric inequality we have
\begin{equation}\label{f:endpoint-lower}
F(t)[-\mu'(t)]^{p-1}
\ge
P(U_t)^p
\ge
n^p\omega_n^{p/n}
\mu(t)^{p(n-1)/n}
\end{equation}
for almost every $t\in(0,1)$.

By the area-formula argument recalled above,
\eqref{f:endpoint-lower} holds at $t=v(r)$
for almost every $r$ such that $v'(r)<0$.
At these points, \eqref{f:distribution-v} gives
\[
\mu(v(r))=\omega_n r^n,
\qquad
[-\mu'(v(r))][-v'(r)]
=n\omega_n r^{n-1}.
\]
Consequently, 
evaluating \eqref{f:endpoint-lower} at $t=v(r)$ yields
\begin{equation}\label{f:endpoint-lower-r}
F(v(r))
\ge
n\omega_n r^{n-1}[-v'(r)]^{p-1}.
\end{equation}

On the other hand, by \eqref{f:flux-identity} we have
\begin{equation}\label{f:endpoint-flux}
F(t)
=
p\lambda_{p,p-2}(\Omega)
\int_{U_t}u^{p-1}\,dx.
\end{equation}

Since $u^*(x)=v(|x|)$, equimeasurability and
$\mu(v(r))=\omega_nr^n$ give
\[
\int_{U_{v(r)}}u^{p-1}\,dx
=
n\omega_n
\int_0^r
\rho^{n-1}v(\rho)^{p-1}\,d\rho.
\]
Hence, 
\begin{equation}\label{f:endpoint-upper-r}
F(v(r))
=
p\lambda_{p,p-2}(\Omega)n\omega_n
\int_0^r
\rho^{n-1}v(\rho)^{p-1}\,d\rho.
\end{equation}
Combining \eqref{f:endpoint-lower-r} and
\eqref{f:endpoint-upper-r}, we obtain
\[
r^{n-1}[-v'(r)]^{p-1}
\le
p\lambda_{p,p-2}(\Omega)
\int_0^r
\rho^{n-1}v(\rho)^{p-1}\,d\rho.
\]
At points where $v'(r)=0$ the same inequality is immediate. Thus
\eqref{f:key-gamma} holds also for $\gamma=p-2$, since in this case
we have 
$\delta=n-1$ and 
$p'(\gamma+1)=p$. 
 \qed

\section{Rigidity}\label{s:rigidity}

We now analyze the equality case. The point is that equality in the
eigenvalue comparison propagates backwards through the proof of
\eqref{f:key-gamma}, eventually forcing equality in the Euclidean
isoperimetric inequality for almost every superlevel set.

\begin{proposition}[Rigidity of the superlevel sets]\label{p:rigidity-levels}
Assume that
\begin{equation}\label{f:equality-eigenvalues}
\eigeng(\Omega)=\eigeng(B_R).
\end{equation}
Then, for almost every $t\in(0,1)$,
\begin{equation}\label{f:iso-equality}
P(U_t)
=
n\omega_n^{1/n}|U_t|^{(n-1)/n}.
\end{equation}
Consequently, for almost every $t\in(0,1)$, the set $U_t$ agrees,
up to a set of Lebesgue measure zero, with a ball.
\end{proposition}

\begin{proof}
By Proposition~\ref{p:1d-comparison}, assumption
\eqref{f:equality-eigenvalues} implies that $v=z$ and that
\eqref{f:key-gamma} is an equality for almost every $r\in(0,R)$.
In particular,
\[
v'(r)=z'(r)<0
\qquad\text{for every }r\in(0,R).
\]

Suppose first that $\gamma<p-2$. We trace the equality backwards
through the three steps in the proof of \eqref{f:key-gamma}. 
In Step~3 we obtained \eqref{f:key-gamma} by proving the two inequalities
\begin{equation}\label{f:K-sandwich}
r^\delta[-v'(r)]^{\gamma+1}
\leq
[(\gamma+2)\eigeng(\Omega)]^{\frac{\delta}{m(n-1)}} J(r)
\leq
p' (\gamma+1) \eigeng(\Omega)
\int_0^r\rho^\delta v(\rho)^{\gamma+1}\,d\rho 
\end{equation}
(see \eqref{f:first-K} and \eqref{f:second-K}).
Since the extreme terms in
\eqref{f:K-sandwich} coincide by the equality form of
\eqref{f:key-gamma}, both inequalities in \eqref{f:K-sandwich}
are equalities for almost every $r\in(0,R)$.

In particular, equality holds in \eqref{f:first-K}. Recalling the
definition of $J(r)$, this is equivalent to equality in
\eqref{f:I-ineq}, and hence in the inequality~\eqref{f:intermediate-radial} obtained in Step~2:
\begin{equation}\label{f:intermediate-radial-equality}
r^{m(n-1)}[-v'(r)]^{m(p-1)} = (\gamma+2)\eigeng(\Omega)
\int_0^r\rho^{m(n-1)} v(\rho)^{\gamma+1}[-v'(\rho)]^{1-m}\,d\rho
\end{equation}
for almost every $r\in(0,R)$.

On the other hand, Step~2 was obtained by combining
\eqref{f:Y-lower-r} and \eqref{f:Y-upper-r}. Thus, for almost every
$r\in(0,R)$,
\[
\begin{split}
(n\omega_n)^m & r^{m(n-1)}[-v'(r)]^{m(p-1)}
\leq Y(v(r))
\\ & \leq
(\gamma+2)\eigeng(\Omega)(n\omega_n)^m
\int_0^r
\rho^{m(n-1)} v(\rho)^{\gamma+1}[-v'(\rho)]^{1-m}\,d\rho.
\end{split}
\]
By \eqref{f:intermediate-radial-equality}, the two extreme terms
coincide. Hence equality holds throughout this chain.

Since $v=z$ and $v'(r)<0$ for every $r\in(0,R)$,
the area-formula argument preceding \eqref{f:inverse}
ensures that all the level-set inequalities used
in the derivation of \eqref{f:Y-lower-r} hold at
$t=v(r)$ for almost every $r\in(0,R)$.
On this common full-measure set, equality in
\eqref{f:Y-lower-r} forces equality throughout
the chain \eqref{f:Y-perimeter}. In particular, 
\[
P(U_{v(r)})^{\gamma+2}
=
n^{\gamma+2}\omega_n^{(\gamma+2)/n}
\mu(v(r))^{(n-1)(\gamma+2)/n},
\]
and therefore
\begin{equation}\label{f:iso-equality-v}
P(U_{v(r)})
=
n\omega_n^{1/n}
\mu(v(r))^{(n-1)/n}
\end{equation}
for almost every $r\in(0,R)$.

Suppose now that $\gamma=p-2$. From
\eqref{f:endpoint-lower-r} and \eqref{f:endpoint-upper-r} we have
\[
n\omega_n r^{n-1}[-v'(r)]^{p-1}
\le
F(v(r))
=
p\lambda_{p,p-2}(\Omega)n\omega_n
\int_0^r
\rho^{n-1}v(\rho)^{p-1}\,d\rho
\]
for almost every $r\in(0,R)$. Since \eqref{f:key-gamma} is an
equality, the two extreme terms coincide. Hence equality holds in
\eqref{f:endpoint-lower-r}. Tracing this equality back through
\eqref{f:endpoint-lower}, we obtain
\[
P(U_{v(r)})^p
=
n^p\omega_n^{p/n}
\mu(v(r))^{p(n-1)/n}
\]
for almost every $r\in(0,R)$, and hence again
\eqref{f:iso-equality-v}.

Thus, in either case, \eqref{f:iso-equality-v} holds
outside a null set $N\subset(0,R)$.
Since $v=z$ is strictly decreasing and belongs to
$AC_{\rm loc}(0,R)$, it maps null sets into null
sets, by the Lusin property (N) applied on a
countable exhaustion of $(0,R)$. Hence $v(N)$
has measure zero. Moreover, $v$ maps $(0,R)$
bijectively onto $(0,1)$. Therefore,
\eqref{f:iso-equality} holds for almost every
$t\in(0,1)$. 

The equality characterization in the Euclidean isoperimetric inequality
for sets of finite perimeter then implies that $U_t$ agrees, up to a
set of Lebesgue measure zero, with a ball for almost every $t\in(0,1)$;
see \cite[Chapter~12]{Maggi}.
\end{proof}

\begin{proposition}[Rigidity of the domain]\label{p:rigidity-domain}
Under the assumption \eqref{f:equality-eigenvalues}, the domain
$\Omega$ is a ball.
\end{proposition}

\begin{proof}
Choose a sequence $t_j\downarrow0$ such that \eqref{f:iso-equality}
holds for every $j$. Since
\[
U_{t_j}\uparrow\{u>0\}=\Omega,
\]
we also have
$
|U_{t_j}|\uparrow|\Omega|.
$
By Proposition~\ref{p:rigidity-levels}, for each $j$ there exist
$c_j\in\R^n$ and $r_j>0$ such that
\[
U_{t_j}=B(c_j,r_j)
\qquad\text{up to a null set},
\]
and
\[
\omega_n r_j^n=|U_{t_j}|.
\]

Since the superlevel sets are nested, for $k>j$ the corresponding balls
satisfy
\[
B(c_j,r_j)\subset B(c_k,r_k).
\]
Indeed, the inclusion holds up to a null set, and any failure of the
actual inclusion between two open balls would produce a set of positive
measure. Hence
\[
|c_k-c_j|\le r_k-r_j.
\]

Since $|U_{t_j}|\uparrow|\Omega|=\omega_nR^n$, we have $r_j\uparrow R$.
It follows that $(c_j)$ is a Cauchy sequence; let $c_j\to c$. The
increasing union of the balls $B(c_j,r_j)$ then agrees with $B(c,R)$.
On the other hand, positivity of the principal eigenfunction gives
\[
\Omega=\{u>0\}=\bigcup_{j=1}^\infty U_{t_j}.
\]
Therefore $\Omega$ agrees almost everywhere with $B(c,R)$.
Since $\Omega$ is a Lipschitz domain, every point of
$\partial\Omega$ has positive upper density of both $\Omega$
and its complement. Consequently, the set of points of
density one of $\Omega$ is precisely $\Omega$.
The same property holds for the ball $B(c,R)$.
Since sets that agree almost everywhere have the same
density-one points, we conclude that
$\Omega=B(c,R)$. 
\end{proof}

\bigskip\bigskip

\section*{
Conflict of interest and data availability.} 
The authors declare that they have no conflict of interest.
No datasets were generated or analyzed during the current study.


\bigskip\bigskip

\begin{thebibliography}{99}


\bibitem{AlcantaraSantos2026}
C.~Alcantara and M.~S.~Santos,
\emph{Regularity for the normalized $p$-Laplacian equation with an arbitrary degeneracy law},
Adv. Calc. Var. \textbf{19} (2026), no.~2, 271--288.

\bibitem{AlcantaraSantosUrbano2026}
C.~Alcantara, M.~Santos and J.~M.~Urbano,
\emph{Gradient regularity for a class of singular or degenerate elliptic equations},
J. Differential Equations \textbf{458} (2026), 114040.

\bibitem{AllegrettoHuang1998}
W.~Allegretto and Y.~X.~Huang,
A Picone's identity for the $p$-Laplacian and applications,
\emph{Nonlinear Anal.} \textbf{32} (1998), no.~7,
819--830.

\bibitem{AttouchiRuosteenoja}
A.~Attouchi and E.~Ruosteenoja,
\emph{Remarks on regularity for $p$-Laplacian type equations in non-divergence form},
J. Differential Equations \textbf{265} (2018), no.~5, 1922--1961.

\bibitem{Bhattacharya1999}
T.~Bhattacharya,
\emph{A proof of the Faber--Krahn inequality for the first eigenvalue of the $p$-Laplacian},
Ann. Mat. Pura Appl. (4) \textbf{177} (1999), 225--240.

\bibitem{BirindelliDemengel2007}
I.~Birindelli and F.~Demengel,
\emph{Eigenvalue, maximum principle and regularity for fully non linear homogeneous operators},
Commun. Pure Appl. Anal. \textbf{6} (2007), no.~2, 335--366.

\bibitem{BirindelliDemengel2009}
I.~Birindelli and F.~Demengel,
\emph{Eigenvalue and Dirichlet problem for fully-nonlinear operators in non-smooth domains},
J. Math. Anal. Appl. \textbf{352} (2009), no.~2, 822--835.

\bibitem{BirindelliLeoni2014}
I.~Birindelli and F.~Leoni,
\emph{Symmetry minimizes the principal eigenvalue:
an example for the Pucci's sup operator},
Math. Res. Lett. \textbf{21} (2014), no.~5,
953--967.

\bibitem{BucurDaners2010}
D.~Bucur and D.~Daners,
\emph{An alternative approach to the Faber--Krahn inequality for Robin problems},
Calc. Var. Partial Differential Equations \textbf{37} (2010), 75--86.

\bibitem{CrandallIshiiLions}
M.~G.~Crandall, H.~Ishii and P.-L.~Lions,
\emph{User's guide to viscosity solutions of second order partial differential equations},
Bull. Amer. Math. Soc. (N.S.) \textbf{27} (1992), no.~1, 1--67.

\bibitem{CrastaFragala2020}
G.~Crasta and I.~Fragal\`a,
\emph{The Brunn--Minkowski inequality for the principal eigenvalue of fully nonlinear homogeneous elliptic operators},
Adv. Math. \textbf{359} (2020), 106855.

\bibitem{CrastaFragalaKawohl}
G.~Crasta, I.~Fragal\`a and B.~Kawohl,
\emph{On the first eigenvalue of the normalized $p$-Laplacian},
Proc. Amer. Math. Soc. \textbf{148} (2020), 577--590.

\bibitem{Daners2006}
D.~Daners,
\emph{A Faber--Krahn inequality for Robin problems in any space dimension},
Math. Ann. \textbf{335} (2006), 767--785.

\bibitem{Daners2011}
D.~Daners,
\emph{Krahn's proof of the Rayleigh conjecture revisited},
Arch. Math. (Basel) \textbf{96} (2011), no.~2, 187--199.

\bibitem{EmbrechtsHofert}
P.~Embrechts and M.~Hofert,
\emph{A note on generalized inverses},
Math. Methods Oper. Res. \textbf{77} (2013), no.~3, 423--432.



\bibitem{Faber1923}
G.~Faber,
\emph{Beweis dass unter allen homogenen Membranen von gleicher Fl\"ache
und gleicher Spannung die kreisf\"ormige den tiefsten Grundton gibt},
Sitzungsber. Bayer. Akad. Wiss. (1923), 169--172.

\bibitem{HamelNadirashviliRuss2007}
F.~Hamel, N.~Nadirashvili and E.~Russ,
\emph{Some isoperimetric problems for the principal eigenvalues of second-order elliptic operators in divergence form},
C. R. Math. Acad. Sci. Paris \textbf{344} (2007), no.~3, 169--174.

\bibitem{HamelNadirashviliRuss2011}
F.~Hamel, N.~Nadirashvili and E.~Russ,
\emph{Rearrangement inequalities and applications to isoperimetric problems for eigenvalues},
Ann. of Math. (2) \textbf{174} (2011), no.~2, 647--755.

\bibitem{ImbertSilvestre2013}
C.~Imbert and L.~Silvestre,
\emph{$C^{1,\alpha}$ regularity of solutions of some degenerate fully non-linear elliptic equations},
Adv. Math. \textbf{233} (2013), 196--206.


\bibitem{JulinJuutinen2012}
V.~Julin and P.~Juutinen,
\emph{A new proof for the equivalence of weak and viscosity solutions for the $p$-Laplace equation},
Commun. Partial Differential Equations \textbf{37} (2012), no.~5, 934--946.


\bibitem{Kawohl}
B.~Kawohl,
\emph{Rearrangements and Convexity of Level Sets in PDE},
Lecture Notes in Mathematics, vol.~1150, Springer-Verlag, Berlin, 1985.

\bibitem{KawohlFridman}
B.~Kawohl and V.~Fridman,
\emph{Isoperimetric estimates for the first eigenvalue of the $p$-Laplace operator and the Cheeger constant},
Comment. Math. Univ. Carolin. \textbf{44} (2003), no.~4, 659--667.

\bibitem{KawohlNormalized}
B.~Kawohl,
\emph{How I met the normalized $p$-Laplacian $\Delta_p^N$ and what we know now about mean values and concavity properties},
J. Elliptic Parabol. Equ. \textbf{6} (2020), 113--121.

\bibitem{Krahn1925}
E.~Krahn,
\emph{\"Uber eine von Rayleigh formulierte Minimaleigenschaft des Kreises},
Math. Ann. \textbf{94} (1925), 97--100.

\bibitem{Krahn1926}
E.~Krahn,
\emph{\"Uber Minimaleigenschaften der Kugel in drei und mehr Dimensionen},
Acta Comm. Univ. Tartu (Dorpat) A \textbf{9} (1926), 1--44.

\bibitem{Maggi}
F.~Maggi,
\emph{Sets of Finite Perimeter and Geometric Variational Problems},
Cambridge Studies in Advanced Mathematics, vol.~135, Cambridge University Press, Cambridge, 2012.

\bibitem{ManfrediParviainenRossi}
J.~J.~Manfredi, M.~Parviainen and J.~D.~Rossi,
\emph{An asymptotic mean value characterization for $p$-harmonic functions},
Proc. Amer. Math. Soc. \textbf{138} (2010), no.~3, 881--889.

\bibitem{MedinaOchoa}
M.~Medina and P.~Ochoa,
\emph{On viscosity and weak solutions for non-homogeneous $p$-Laplace equations},
Adv. Nonlinear Anal. \textbf{8} (2019), no.~1, 468--481.

\bibitem{PeresSheffield}
Y.~Peres and S.~Sheffield,
\emph{Tug-of-war with noise: a game-theoretic view of the $p$-Laplacian},
Duke Math. J. \textbf{145} (2008), no.~1, 91--120.

\bibitem{Talenti1976}
G.~Talenti,
\emph{Best constant in Sobolev inequality},
Ann. Mat. Pura Appl. (4) \textbf{110} (1976), 353--372.


\bibitem{WangYinJiang2025}
J.~Wang, Y.~Yin and F.~Jiang,
\emph{Regularity of solutions to degenerate normalized $p$-Laplacian equation with general variable exponents},
Potential Anal. \textbf{63} (2025), 1963--2000.

\end{thebibliography}
\end{document}